\documentclass[11pt,a4paper]{article}
\usepackage[affil-it]{authblk}
\usepackage[utf8]{inputenc}
\usepackage{amsmath, amssymb, mathrsfs, amsthm, amsfonts}
\usepackage{latexsym}
\usepackage[pdftex]{hyperref}
\usepackage{a4wide}
\usepackage{color}

\newtheorem{theorem}{Theorem}[section]

\newtheorem{lemma}[theorem]{Lemma}
\newtheorem{proposition}[theorem]{Proposition}

\theoremstyle{definition}
\newtheorem{definition}[theorem]{Definition}
\theoremstyle{remark}

\title{The Tutte Polynomial of Symmetric Hyperplane Arrangements}

\author{Hery Randriamaro
\thanks{This research was funded by my mother \\
Lot II B 32 bis Faravohitra, 101 Antananarivo, Madagascar \\
e-mail: \texttt{hery.randriamaro@gmail.com}}}

\begin{document}

\maketitle

\begin{abstract}
\noindent The Tutte polynomial is originally a bivariate polynomial which enumerates the colorings of a graph and of its dual graph. Ardila extended in 2007 the definition of the Tutte polynomial on the real hyperplane arrangements. He particularly computed the Tutte polynomials of the hyperplane arrangements associated to the classical Weyl groups. Those associated to the exceptional Weyl groups were computed by De Concini and Procesi one year later. This article has two objectives: On one side, we extend the Tutte polynomial computing to the complex hyperplane arrangements. On the other side, we introduce a wider class of hyperplane arrangements which is that of the symmetric hyperplane arrangements. Computing the Tutte polynomial of a symmetric hyperplane arrangement permits us to deduce the Tutte polynomials of some hyperplane arrangements, particularly of those associated to the imprimitive reflection groups.

\bigskip 


\smallskip

\end{abstract}

\section{Introduction}

\subsection*{The Tutte Polynomial}

\noindent The concept of coloring originates from the map coloring problem. One assigns a vertex to each territory, and connect two vertices if and only if the corresponding territories are adjacent. One associates in this way a graph to any map of countries. A graph coloring corresponds to a way of coloring the vertices, so that two connected vertices are differently colored. And the chromatic polynomial is a graph polynomial which counts the number of graph colorings. In 1954, Tutte obtained a two--variable polynomial from which the chromatic polynomial of a graph, and that of its dual graph can be deduced \cite[§~3]{Tu}: it is the Tutte polynomial. Although this polynomial reveals more of the internal structure of the graph like its number of forests, of spanning subgraphs, and of acyclic orientations. Crapo showed also that it determines the number of subgraphs with any stated number of edges and any stated rank \cite[Theorem~1]{Cr}. In this article, we study the Tutte polynomial on objects generalizing the finite simple graphs, which are the hyperplane arrangements. Determining the Tutte polynomial of a graph is indeed equivalent to determining the Tutte polynomial of an associated hyperplane arrangement called graphic arrangement.

\smallskip

\noindent Let $z = (z_1, \dots, z_n)$ be a vector variable of the Hermitian space $\mathbb{C}^n$, and $c_1, \dots, c_n, d$ $n+1$ complex coefficients such that $(c_1, \dots, c_n) \neq (0, \dots, 0)$. A hyperplane $H$ of $\mathbb{C}^n$ is an affine subspace $H := \{z \in \mathbb{C}^n \ |\ c_1 z_1 + \dots + c_n z_n = d\}$. For simplicity, we just write $$H := \{c_1 z_1 + \dots + c_n z_n = d\}.$$

\noindent A hyperplane arrangement $\mathcal{A}$ in $\mathbb{C}^n$ is a finite set of hyperplanes. Denote by $\cap \mathcal{A}$ the subspace $$\cap \mathcal{A} := \bigcap_{H \in \mathcal{A}} H.$$ 
One says that the arrangement $\mathcal{A}$ is said central if $\cap \mathcal{A} \neq \emptyset$.\\
A subarrangement of $\mathcal{A}$ is a subset of $\mathcal{A}$. The rank function $r$ is defined for each central subarrangement $\mathcal{B}$ of $\mathcal{A}$ by $$r(\mathcal{B}) := n - \dim \cap \mathcal{B}.$$ This function is extended to the function $r: 2^{\mathcal{A}} \rightarrow \mathbb{N}$ by defining the rank of a noncentral subarrangement $\mathcal{B}$ to be the largest rank of a central subarrangement of $\mathcal{B}$.

\noindent Take two variables $x$ and $y$. The \emph{Tutte polynomial} of a hyperplane arrangement $\mathcal{A}$ is
$$T_{\mathcal{A}}(x,y) := \sum_{\substack{\mathcal{B} \subseteq \mathcal{A} \\ \text{central}}} (x-1)^{r(\mathcal{A})-r(\mathcal{B})} (y-1)^{|\mathcal{B}|-r(\mathcal{B})}.$$

\noindent At various points and lines of the $(x,y)$-plane, the Tutte polynomial evaluates to quantities studied in diverse fields of Mathematics and Physics. For examples,
\begin{itemize}
\item[$\bullet$] Along the hyperbola $xy=1$, the Tutte polynomial specializes to the Jones polynomial of an associated alternating knot \cite[§~3]{Th}. The Jones polynomial is an invariant of an oriented knot which assigns to each oriented knot a Laurent polynomial in the variable $t^{\frac{1}{2}}$ with integer coefficients.
\item[$\bullet$] For any positive integer q, along the hyperbola $(x-1)(y-1)=q$, the Tutte polynomial specializes to the partition function of the $q$-state Potts model \cite[§~I]{MeWe}. By studying the Potts model, one may gain insight into the behavior of ferromagnets, and certain other phenomena of solid-state physics.
\item[$\bullet$] At $x=1$, the Tutte polynomial specializes to the all-terminal reliability polynomial studied in network theory \cite[Proposition~3.3]{Br}. The reliability polynomial gives the probability that every pair of vertices of the graph remains connected after edges fail. 
\end{itemize}

\noindent Let $\mathcal{A}$ be a hyperplane arrangement in $\mathbb{C}^n$, and $\displaystyle M_{\mathcal{A}} = \mathbb{C}^n \setminus \bigcup_{H \in \mathcal{A}} H$ the complement of its variety. Then, the Poincaré polynomial of the cohomology ring of $M_{\mathcal{A}}$ is given by \cite{OrSo}
$$\sum_{k \geq 0} \mathrm{rank}\, H^k(M_{\mathcal{A}}, \mathbb{Z})\, q^k = (-1)^{r(\mathcal{A})} q^{n- r(\mathcal{A})} T_{\mathcal{A}}(1-q,0).$$ 

\noindent If $\mathcal{A}$ is a hyperplane arrangement in $\mathbb{R}^n$, then the number of regions into which $\mathcal{A}$ dissects $\mathbb{R}^n$ is equal to $|T_{\mathcal{A}}(2,0)|$ \cite[Theorem 2.5]{Sta}.

\smallskip

\noindent Take two other variables $q,t$, and consider a simple transformation of the Tutte polynomial, called the \emph{coboundary polynomial} of the hyperplane arrangement $\mathcal{A}$, defined by
$$\bar{\chi}_{\mathcal{A}}(q,t) := \sum_{\substack{\mathcal{B} \subseteq \mathcal{A} \\ \text{central}}} q^{r(\mathcal{A}) -r(\mathcal{B})} (t-1)^{|\mathcal{B}|}.$$
Since $\displaystyle T_{\mathcal{A}}(x,y) = \frac{\bar{\chi}_{\mathcal{A}}\big((x-1)(y-1), y \big)}{(y-1)^{r(\mathcal{A})}}$, computing the coboundary polynomial of a hyperplane arrangement is equivalent to computing its Tutte polynomial.\\
Recall that the hyperplane arrangements associated to the classical Weyl groups are
\begin{align*}
& (A_{n-1}) \quad \mathcal{A}_{A_{n-1}} = \big\{ \{z_i - z_j = 0\} \big\}_{1 \leq i < j \leq n},\\
& (B_n, C_n) \ \mathcal{A}_{B_n} = \big\{ \{z_i \pm z_j = 0\} \big\}_{1 \leq i < j \leq n} \sqcup \big\{ \{z_i = 0\} \big\}_{i \in [n]},\\
& (D_n) \qquad \mathcal{A}_{D_n} = \big\{ \{z_i \pm z_j = 0\} \big\}_{1 \leq i < j \leq n}.
\end{align*}
Ardila proved that
\begin{align*}
& \text{\cite[Theorem 4.1]{Ar}} \quad 1 + q \sum_{n \geq 1} \bar{\chi}_{\mathcal{A}_{A_n}}(q,t) \frac{x^n}{n!} = \Big(\sum_{n \in \mathbb{N}} t^{\binom{n}{2}} \frac{x^n}{n!}\Big)^q,\\
& \text{\cite[Theorem 4.2]{Ar}} \quad \sum_{n \in \mathbb{N}} \bar{\chi}_{\mathcal{A}_{B_n}}(q,t) \frac{x^n}{n!} = \Big(\sum_{n \in \mathbb{N}} t^{n^2} \frac{x^n}{n!}\Big) \Big(\sum_{n \in \mathbb{N}} 2^n t^{\binom{n}{2}} \frac{x^n}{n!}\Big)^{\frac{q-1}{2}},\\
& \text{\cite[Theorem 4.3]{Ar}} \quad \sum_{n \in \mathbb{N}} \bar{\chi}_{\mathcal{A}_{D_n}}(q,t) \frac{x^n}{n!} = \Big(\sum_{n \in \mathbb{N}} t^{n(n-1)} \frac{x^n}{n!}\Big) \Big(\sum_{n \in \mathbb{N}} 2^n t^{\binom{n}{2}} \frac{x^n}{n!}\Big)^{\frac{q-1}{2}}.
\end{align*}

\noindent De Concini and Procesi computed the Tutte polynomials of the hyperplane arrangements associated to the exceptional Weyl groups \cite[3.4. Exceptional types]{DePr}.

\subsection*{Symmetric Hyperplane Arrangement}

\noindent \emph{We aim to give a more general result by introducing the symmetric hyperplane arrangements, and the $\zeta_m$--coboundary polynomial of a hyperplane arrangement: More exactly, we compute the $\zeta_m$--coboundary polynomial of a symmetric hyperplane arrangement. That permits us to deduce the Tutte polynomials of hyperplane arrangements having symmetric structure, particularly those of the hyperplane arrangements associated to the imprimitive reflection groups.} 

\smallskip

\noindent Let $\mathbb{U}_m$ be the group of all $m^{\text{th}}$ roots of unity. For a primitive $m^{\text{th}}$ root of unity $\zeta_m$, denote by $l_m$ the dimension of the module $\mathbb{Z}[\zeta_m]$:
\begin{itemize}
\item[$\bullet$] If $m$ is odd, $l_m = m$,
\item[$\bullet$] If $m$ is even, $\displaystyle l_m = \frac{m}{2}$.
\end{itemize}

\noindent It is useful for us to generalize the coboundary polynomial into what we call \emph{$\zeta_m$--coboundary polynomial} of a hyperplane arrangement $\mathcal{A}$ defined by
$$\bar{\chi}_{\mathcal{A}}^{(\zeta_m)}(q,t) := \sum_{\substack{\mathcal{B} \subseteq \mathcal{A} \\ \text{central}}} q^{l_m(r(\mathcal{A}) -r(\mathcal{B}))} (t-1)^{|\mathcal{B}|}.$$
Since $\displaystyle T_{\mathcal{A}}(x,y) = \frac{\bar{\chi}_{\mathcal{A}}^{(\zeta_m)}\big((x-1)^{\frac{1}{l_m}}(y-1)^{\frac{1}{l_m}}, y \big)}{(y-1)^{r(\mathcal{A})}}$, computing the $\zeta_m$--coboundary polynomial of a hyperplane arrangement is equivalent to computing its Tutte polynomial.

\smallskip

\noindent An element $\sigma$ of the symmetric group $\mathfrak{S}_n$ acts on the hyperplane $H = \{c_1 z_1 + \dots + c_n z_n = d\}$ in $\mathbb{C}^n$ by
$$\sigma \cdot H := \{c_1 z_{\sigma(1)} + \dots + c_n z_{\sigma(n)} = d\}.$$
More generally, an element $\sigma$ of the symmetric group $\mathfrak{S}_n$ acts on the hyperplane arrangement $\mathcal{A}= \{H_1, \dots, H_m\}$ by
$$\sigma \cdot \mathcal{A} := \{\sigma \cdot H_1, \dots, \sigma \cdot H_m\}.$$

\begin{definition}
A \emph{symmetric hyperplane arrangement} or a \emph{\textsc{SH}--arrangement} is a hyperplane arrangement $\mathcal{A}$ in $\mathbb{C}^n$ such that, for every permutation $\sigma$ in $\mathfrak{S}_n$, we have $\sigma \cdot \mathcal{A} = \mathcal{A}$.
\end{definition} 

\noindent The hyperplane arrangements associated to the classical Weyl groups are symmetric hyperplane arrangements. Another is the hyperplane arrangement $\mathcal{I}_n$ in $\mathbb{C}^n$ defined, for $n \geq 2$, by $$\mathcal{I}_n := \big\{\{z_i = 0\}\big\}_{i \in [n]} \sqcup \big\{\{z_i = 1\}\big\}_{i \in [n]} \sqcup  \big\{\{z_i + z_j = 1\}\big\}_{1 \leq i < j \leq n}.$$

\noindent We particularly study the special case of the colored symmetric hyperplane arrangements. Recall that the colored permutation group of $n$ elements with $m$ colors is the wreath product $\mathbb{U}_m \wr \mathfrak{S}_n$, and one represents one of its elements $\pi$ by 
$$\pi = \left( \begin{array}{cccc} 1 & 2 & \dots & n\\
\xi_1 \sigma(1) & \xi_2 \sigma(2) & \dots & \xi_n \sigma(n) \end{array} \right)\ \text{with}\ \sigma \in \mathfrak{S}_n\ \text{and}\ \xi_i \in \mathbb{U}_m.$$

\noindent The colored permutation $\pi$ of $\mathbb{U}_m \wr \mathfrak{S}_n$ acts on the hyperplane $H = \{c_1 z_1 + \dots + c_n z_n = d\}$ in $\mathbb{C}^n$ by
$$\pi \cdot H := \{c_1 \xi_1 z_{\sigma(1)} + \dots + c_n \xi_n z_{\sigma(n)} = d\}.$$
More generally, the colored permutation $\pi$ acts on the hyperplane arrangement $\mathcal{A} = \{H_1, \dots, H_m\}$ in $\mathbb{C}^n$ by $$\pi \cdot \mathcal{A} := \{\pi \cdot H_1, \dots, \pi \cdot H_m\}.$$

\begin{definition}
A \emph{colored symmetric hyperplane arrangement} or a \emph{\textsc{CSH}--arrangement} is a hyperplane arrangement $\mathcal{A}$ in $\mathbb{C}^n$ such that, for every colored permutation $\pi$ of $\mathbb{U}_m \wr \mathfrak{S}_n$, we have $\pi \cdot \mathcal{A} = \mathcal{A}$.
\end{definition}

\noindent An example of colored symmetric hyperplane arrangement is the hyperplane arrangement associated to the imprimitive reflection group $G(m,p,n)$ defined by
$$\mathcal{G}_{m,p,n} := \big\{\{z_i = 0\}\big\}_{i \in [n]} \sqcup \big\{\{z_i - \xi z_j = 0\}\big\}_{\substack{1 \leq i < j \leq n \\ \xi \in \mathbb{U}_m}}.$$

\subsection*{Structure of this Article}

\noindent Consider a finite set $\mathbb{S}$, and take an element $v = (v_1, \dots, v_n)$ of $\mathbb{S}^n$.

\noindent Define the function $\mathtt{S}(v) := \{v_1, \dots, v_n\}$, and extend it to every subset $M$ of $\displaystyle \bigsqcup_{n \in \mathbb{N}} \mathbb{S}^n$ by
$$\mathtt{S}(M) := \bigcup_{v \in M} \mathtt{S}(v).$$

\noindent For an element $t$ in $\mathbb{S}$, denote by $\mathtt{o}_t(v)$ its number of occurrences in $v$. Define the integer sequence $\mathtt{c}(v)$ indexed by $\mathbb{S}$ by
$$\mathtt{c}(v) := \big(\mathtt{o}_t(v)\big)_{t \in \mathbb{S}}.$$

\noindent Denote by $\mathbb{N}^{\mathbb{S}}$ the set of integer sequences $\mathsf{a} = (\mathsf{a}_t)_{t \in \mathbb{S}}$ indexed by $\mathbb{S}$. We write
$$\binom{n}{\mathsf{a}} \ \text{for the multinomial} \ \binom{n}{(\mathsf{a}_t)_{t \in \mathbb{S}}},$$
and $\|\mathsf{a}\|$ for the norm $$\|\mathsf{a}\| := \sum_{t \in \mathbb{S}} \mathsf{a}_t.$$

\noindent Define the function $\mathtt{f}: \mathbb{N}^{\mathbb{S}} \times \mathbb{S}^n \rightarrow \mathbb{N}$ by
$$\mathtt{f}(\mathsf{a}, v) := \prod_{t \in \mathtt{S}(v)} \binom{\mathsf{a}_t}{\mathtt{o}_t(v)}.$$
Extend that function to every subset $M$ of $\displaystyle \bigsqcup_{n \in \mathbb{N}} \mathbb{S}^n$ by
$$\mathtt{f}(\mathsf{a}, M) := \sum_{v \in M} \mathtt{f}(\mathsf{a}, v).$$

\begin{definition}
A $\mathbb{Z}[\zeta_m]$--hyperplane is a hyperplane $H = \{c_1 z_1 + \dots + c_n z_n = d\}$ in $\mathbb{C}^n$ whose coefficients $c_1, \dots, c_n, d$ are in $\mathbb{Z}[\zeta_m]$.
\end{definition}

\noindent \textsl{From now on, unless explicitly stated, every hyperplane considered in this article is a $\mathbb{Z}[\zeta_m]$--hyperplane, and every hyperplane arrangement an arrangement of $\mathbb{Z}[\zeta_m]$--hyperplanes.}

\smallskip

\noindent $\mathbb{F}_q$ is the usually field of order $q$. In \textbf{Section \ref{Symmetric}}, we define what we call solution indice partition $\mathrm{M}_1, \dots, \mathrm{M}_s$ of a \textsc{SH}--arrangement $\mathcal{A}$, and prove that its $\zeta_m$--coboundary polynomial is
$$\bar{\chi}_{\mathcal{A}}^{(\zeta_m)}(q,t) = q^{l_m(r(\mathcal{A})-n)} \sum_{\substack{\mathsf{a} \in \mathbb{N}^{\mathbb{F}_q[\zeta_m]} \\ \|\mathsf{a}\| = n}} \binom{n}{\mathsf{a}} t^{\sum_{j=1}^s \mathtt{f}(\mathsf{a}, \mathrm{M}_j)}.$$
One directly obtains the coboundary polynomials of the hyperplane arrangements associated to the classical Weyl groups from it. As computing example, we show that
\begin{align*}
\sum_{n \in \mathbb{N}} q \bar{\chi}_{\mathcal{I}_n}(q,t) \frac{x^n}{n!}\ =\ & \Big( \sum_{n \in \mathbb{N}} t^{\binom{n}{2}} \frac{x^n}{n!} \Big) \times \Big( \sum_{n \in \mathbb{N}} \big( \sum_{a+b=n} \binom{n}{a,b} t^{a+b+ab} \big) \frac{x^n}{n!} \Big) \\
& \times \Big( \sum_{n \in \mathbb{N}} \big( \sum_{a+b=n} \binom{n}{a,b} t^{ab} \big) \frac{x^n}{n!} \Big)^{\frac{q-3}{2}}.
\end{align*}

\noindent Take an element $u$ in $\mathbb{S}$ and a positive integer $m$. Define the function $\mathtt{f}_u^{(m)}: \mathbb{N}^{\mathbb{S}} \times \mathbb{S}^n \rightarrow \mathbb{N}$ by
$$\mathtt{f}_u^{(m)}(\mathsf{a}, v) = \left\{ \begin{array}{ll}
m^{n-1} \binom{\mathsf{a}_u}{\mathtt{o}_u(v)} & \text{if}\ \mathtt{o}_u(v) = n \\
m^{\mathtt{o}_u(v)} \mathtt{f}(\mathsf{a}, v) & \text{otherwise} \end{array} \right..$$
Extend that function to every subset $M$ of $\displaystyle \bigsqcup_{n \in \mathbb{N}} \mathbb{S}^n$ by
$$\mathtt{f}_u^{(m)}(\mathsf{a}, M) := \sum_{v \in M} \mathtt{f}_u^{(m)}(\mathsf{a}, v).$$

\noindent In \textbf{Section \ref{Colored}}, we define what we call solution indice partition $\mathrm{M}_{1 \ast}, \dots, \mathrm{M}_{s \ast}$ of a \textsc{CSH}--arrangement $\mathcal{A}$, and prove that its $\zeta_m$--coboundary polynomial is
$$\bar{\chi}_{\mathcal{A}}^{(\zeta_m)}(q,t) = q^{l_m(r(\mathcal{A})-n)} \sum_{\substack{\mathsf{a} \in \mathbb{N}^{\mathbb{F}_q[\zeta_m]_{\ast}} \\ \|\mathsf{a}\| = n}} \binom{n}{\mathsf{a}} m^{n- \mathsf{a}_{\ddot{0}}} t^{\sum_{j=1}^s \mathtt{f}_{\ddot{0}}^{(m)}(\mathsf{a}, \mathrm{M}_{j \ast})}.$$
That permits us to deduce the exponential generating function of the $\zeta_m$--coboundary polynomials of the hyperplane arrangements $\mathcal{G}_{m,p,n}$  
$$\sum_{n \in \mathbb{N}} \bar{\chi}_{\mathcal{G}_{m,p,n}}^{(\zeta_m)}(q,t) \frac{x^n}{n!} = \Big(\sum_{n \in \mathbb{N}} t^{n+m\binom{n}{2}} \frac{x^n}{n!}\Big) \Big(\sum_{n \in \mathbb{N}} m^n t^{\binom{n}{2}} \frac{x^n}{n!}\Big)^{\frac{q^{l_m}-1}{m}}.$$
\emph{The author would like to thank Federico Ardila who suggested him to compute the Tutte polynomials of the hyperplane arrangements associated to the imprimitive reflection groups.}

\smallskip

\noindent But before determining the $\zeta_m$--coboundary polynomials, we have to prove an exponential generating function formula in \textbf{Section \ref{Exponential}} that we need later.\\
We also extend the finite field method to hyperplane arrangements in $\mathbb{C}^n$ in \textbf{Section \ref{Complex}}. This extension is fundamental to compute the $\zeta_m$--coboundary polynomials of a \textsc{SH}--arrangement, and a \textsc{CSH}--arrangement.

\section{Exponential Generating Function} \label{Exponential}

\noindent We briefly recall some properties of the exponential generating functions in this section. We particularly need Proposition \ref{PrPrin} to prove Proposition \ref{PrSy}, and Theorem \ref{ThCoEGF}.

\begin{definition}
Let $x$ be a variable, and $\mathbb{R}[t]$ the polynomial ring with variable $t$. The \emph{exponential generating function} $u(x)$ or EGF of a sequence $(u_n)_{n \in \mathbb{N}}$, $u_n$ belonging to $\mathbb{R}[t]$, is the formal power serie
$$u(x) := \sum_{n \in \mathbb{N}} u_n \frac{x^n}{n!}.$$
\end{definition}

\noindent Take $q$ other EGF's $v^{(1)}(x), \dots, v^{(q)}(x)$ such that $$v^{(i)}(x) = \sum_{n \in \mathbb{N}} v_n^{(i)} \frac{x^n}{n!},\ \text{and}\ u(x) = v^{(1)}(x) \cdot \, \dots \, \cdot v^{(q)}(x).$$
We know from the convolution formula \cite[§~II.2]{FlSe} that $$u_n = \sum_{a_1 + \dots + a_q = n} \binom{n}{a_1, \dots, a_q}\, v_{a_1}^{(1)} \, \dots \, v_{a_q}^{(q)}.$$

\noindent In the next proposition, we have a more general form of the convolution formula.

\begin{proposition} \label{PrPrin}
Let $q \in \mathbb{N}^*$, and $j \in [q]$. Take $q$ vector variables $\mathsf{a}^{(j)} = (a_1^{(j)}, \dots, a_{i_j}^{(j)})$. Consider $q$ multi-indexed sequences of polynomials $\displaystyle \big( p_{\mathsf{a}^{(j)}}^{(j)} \big)_{\mathsf{a}^{(j)} \in \mathbb{N}^{i_j}}$ with $\displaystyle p_{\mathsf{a}^{(j)}}^{(j)} \in \mathbb{R}[t]$.\\ Let $\displaystyle u(x) = \sum_{n \in \mathbb{N}} u_n \frac{x^n}{n!}$ be an EGF such that
$$u_n = \sum_{\| \mathsf{a}^{(1)} \| + \dots + \| \mathsf{a}^{(q)} \| = n} \binom{n}{\mathsf{a}^{(1)} \cdot ... \cdot \mathsf{a}^{(q)}} \, p_{\mathsf{a}^{(1)}}^{(1)} \dots p_{\mathsf{a}^{(q)}}^{(q)}.$$
Then, there exist $q$ EGF $\displaystyle v^{(j)}(x) = \sum_{n \in \mathbb{N}} v_n^{(j)} \frac{x^n}{n!}$ such that $u(x) = v^{(1)}(x) \cdot ... \cdot v^{(q)}(x)$, and
$$v_n^{(j)} = \sum_{\| \mathsf{a}^{(j)} \| = n} \binom{n}{\mathsf{a}^{(j)}}\, p_{\mathsf{a}^{(j)}}^{(j)}.$$
\end{proposition}

\begin{proof}
Take $q$ variables $b_j$ such that $b_j = \| \mathsf{a}^{(j)} \|$. We have
\begin{align*}
u_n & = \sum_{\| \mathsf{a}^{(1)} \| + \dots + \| \mathsf{a}^{(q)} \| = n} \binom{n}{\mathsf{a}^{(1)} \cdot ... \cdot \mathsf{a}^{(q)}} \, p_{\mathsf{a}^{(1)}}^{(1)} \dots p_{\mathsf{a}^{(q)}}^{(q)} \\
& = \sum_{\|\mathsf{a}^{(1)}\| + \dots + \|\mathsf{a}^{(q)}\| = n} n! \, \frac{p_{\mathsf{a}^{(1)}}^{(1)}}{a_1^{(1)}! \, \dots \, a_{i_1}^{(1)}!} \, \dots \, \frac{p_{\mathsf{a}^{(q)}}^{(q)}}{a_1^{(q)}! \, \dots \, a_{i_q}^{(q)}!} \\
& = \sum_{\|\mathsf{a}^{(1)}\| + \dots + \|\mathsf{a}^{(q)}\| = n} \binom{n}{b_1, \dots, b_q} \, \prod_{j=1}^q \binom{b_j}{\mathsf{a}^{(j)}}\, p_{\mathsf{a}^{(j)}}^{(j)} \\
& = \sum_{b_1 + \dots + b_q = n} \binom{n}{b_1, \dots, b_q} \, \prod_{j=1}^q \sum_{\|\mathsf{a}^{(j)}\| = b_j} \binom{b_j}{\mathsf{a}^{(j)}} p_{\mathsf{a}^{(j)}}^{(j)}.
\end{align*}
We deduce from the convolution formula that $u(x) = v^{(1)}(x) \cdot \, \dots \, \cdot v^{(q)}(x)$ with
$$v_{b_j}^{(j)} = \sum_{\| \mathsf{a}^{(j)} \| = b_j} \binom{b_j}{\mathsf{a}^{(j)}}\, p_{\mathsf{a}^{(j)}}^{(j)}.$$
\end{proof}

\section{Extension of the Finite Field Method} \label{Complex}

\noindent The finite field method developed by Ardila \cite[Theorem 3.3]{Ar} reduces the determination of the coboundary polynomial to a counting problem for the case of the real hyperplane arrangements. In this section, we extend this method to the complex hyperplane arrangements, and reduce the determination of the $\zeta_m$--coboundary polynomial to a counting problem. We use it after to compute the $\zeta_m$--coboundary polynomials of the \textsc{SH} and the \textsc{CSH}--arrangements.

\smallskip

\noindent Let $q$ be a prime number, and $\mathbb{F}_q$ the field of integer modulo $q$. The reduction modulo $q$ of the element $\displaystyle k = \sum_{i=0}^{l_m-1} k_i \zeta_m^i$ of the ring $\mathbb{Z}[\zeta_m]$ is $$\check{k} := \sum_{i=0}^{l_m-1} \check{k}_i \zeta_m^i \in \mathbb{F}_q[\zeta_m] \quad \text{with} \quad \check{k}_i \equiv k_i \mod q.$$

\noindent For two elements $\displaystyle \sum_{i_1=0}^{l_m-1} \check{h}_{i_1} \zeta_m^{i_1}$ and $\displaystyle \sum_{i_2=0}^{l_m-1} \check{k}_{i_2} \zeta_m^{i_2}$ of the field $\mathbb{F}_q[\zeta_m]$, the multiplication is naturally defined by
$$\sum_{i_1=0}^{l_m-1} \check{h}_{i_1} \zeta_m^{i_1} \cdot \sum_{i_2=0}^{l_m-1} \check{k}_{i_2} \zeta_m^{i_2} := \sum_{i_1, i_2=0}^{l_m-1} \check{h}_{i_1} \check{k}_{i_2} \zeta_m^{i_1 + i_2}.$$

\noindent Take a hyperplane $H = \{c_1 z_1 + \dots + c_n z_n = d\}$ in $\mathbb{C}^n$, where $c_1, \dots, c_n \in \mathbb{Z}[\zeta_m]$, and a prime number $q$. We say that $H$ reduces correctly over $\mathbb{F}_q[\zeta_m]$ if, modulo $q$, it induces a hyperplane in $\mathbb{F}_q[\zeta_m]^n$ denoted $\check{H} := \{\check{c}_1 \check{z}_1 + \dots + \check{c}_n \check{z}_n = \check{d}\}$.

\begin{definition}
A hyperplane arrangement $\mathcal{A}$ in $\mathbb{C}^n$ with coefficients in $\mathbb{Z}[\zeta_m]$ reduces correctly over $\mathbb{F}_q[\zeta_m]$ if it induces a hyperplane arrangement $\check{\mathcal{A}}$ in $\mathbb{F}_q[\zeta_m]^n$, formed by the induced hyperplanes of $\mathcal{A}$ modulo $q$, and the hyperplane arrangements $\mathcal{A}$ and $\check{\mathcal{A}}$ are isomorphic.
\end{definition}

\noindent To get the isomorphism, the determinants of all minors of the matrix formed by the coefficients of the hyperplanes in $\mathcal{A}$ must be nonzero modulo $q$. Thus, if we choose $q$ to be a prime larger than the coefficients, as elements in $\mathbb{Z}[\zeta_m]$, of all these determinants, we can guarantee to have a correct reduction.

\begin{theorem} \label{ThCo}
Let $\mathcal{A}$ be a hyperplane arrangement in $\mathbb{C}^n$ with coefficients in $\mathbb{Z}[\zeta_m]$, and $\check{\mathcal{A}}$ its reduced hyperplane arrangement in $\mathbb{F}_q[\zeta_m]^n$, for a large enough prime number $q$. Then, $$q^{l_m(n-r(\mathcal{A}))} \bar{\chi}_{\mathcal{A}}^{(\zeta_m)}(q,t) = \sum_{\check{z} \in \mathbb{F}_q[\zeta_m]^n} t^{h(\check{z})},$$ where $h(\check{z})$ is the number of hyperplanes in $\check{\mathcal{A}}$ the vector $\check{z}$ lies on. 
\end{theorem}

\begin{proof}
To prove Theorem \ref{ThCo}, we need both remarks:
\begin{itemize}
\item[(R1)] If $\check{V}$ is a $d$-dimensional submodule of $\mathbb{F}_q[\zeta_m]^n$, then $|\check{V}|=q^{l_m \times d}$.
\item[(R2)] For a strictly positive integer $d$, we have
$$\sum_{I \subseteq [d]}t^{|I|} = \sum_{k=0}^d \binom{d}{k} t^k = (1+t)^d.$$
We can now proceed to the proof of the theorem. Let $H(\check{z})$ the set of hyperplanes in $\check{\mathcal{A}}$ that $\check{z}$ lies on. We have
\begin{align*}
q^{l_m(n-r(\mathcal{A}))} \bar{\chi}_{\mathcal{A}}^{(\zeta_m)}(q,t) & = \sum_{\substack{\mathcal{B} \subseteq \mathcal{A} \\ \text{central}}} q^{l_m(n-r(\mathcal{B}))} (t-1)^{|\mathcal{B}|} \\
& = \sum_{\substack{\mathcal{B} \subseteq \mathcal{A} \\ \text{central}}} q^{l_m \times \dim \cap \mathcal{B}} (t-1)^{|\mathcal{B}|}
= \sum_{\substack{\check{\mathcal{B}} \subseteq \check{\mathcal{A}} \\ \text{central}}} q^{l_m \times \dim \cap \check{\mathcal{B}}} (t-1)^{|\check{\mathcal{B}}|} \\
& = \sum_{\substack{\check{\mathcal{B}} \subseteq \check{\mathcal{A}} \\ \text{central}}} |\cap \check{\mathcal{B}}| (t-1)^{|\check{\mathcal{B}}|} \quad \text{(R1)} \\
& = \sum_{\substack{\check{\mathcal{B}} \subseteq \check{\mathcal{A}} \\ \text{central}}} \sum_{\check{z} \in \cap \check{\mathcal{B}}} (t-1)^{|\check{\mathcal{B}}|}
= \sum_{\check{z} \in \mathbb{F}_q[\zeta_m]^n} \sum_{\check{\mathcal{B}} \subseteq H(\check{z})} (t-1)^{|\check{\mathcal{B}}|}\\
& = \sum_{\check{z} \in \mathbb{F}_q[\zeta_m]^n} t^{h(\check{z})} \quad \text{(R2)}.
\end{align*}
\end{itemize}
\end{proof}

\section{The $\zeta_m$--Coboundary Polynomial of a \textsc{SH}--Arrangement} \label{Symmetric}

\noindent We compute the $\zeta_m$--coboundary polynomial of a \textsc{SH}--arrangement, and deduce the EGF of a sequence of \textsc{SH}--arrangement $\zeta_m$--coboundary polynomials. As example, we apply the results on the hyperplane arrangement $\mathcal{I}_n$. Song computed the characteristic polynomials of the arrangements $\mathcal{I}_2$ and $\mathcal{I}_3$ \cite[§~3]{So}. Remark that, by means of Whitney's theorem \cite[Theorem~2.4]{Sta}, one can express the characteristic polynomial $\chi_{\mathcal{A}}(q)$ of a hyperplane arrangement $\mathcal{A}$ in term of its coboundary polynomial by $\chi_{\mathcal{A}}(q) = q^{n-r(\mathcal{A})} \bar{\chi}_{\mathcal{A}}(q,0)$.

\begin{definition}
Let $\mathcal{A}$ be a \textsc{SH}--arrangement in $\mathbb{C}^n$. We can choose some hyperplanes $H_1, \dots, H_r$ in $\mathcal{A}$ such that
$$\mathcal{A} = \bigsqcup_{i=1}^r \big\{\sigma \cdot H_i \big\}_{\sigma \in \mathfrak{S}_n}.$$
The hyperplanes $H_1, \dots, H_r$ are called representative hyperplanes for $\mathcal{A}$.
\end{definition}

\noindent Recall that $\mathcal{I}_n = \big\{\sigma \cdot \{z_1 = 0\}\big\}_{\sigma \in \mathfrak{S}_n} \sqcup \big\{\sigma \cdot \{z_1 = 1\}\big\}_{\sigma \in \mathfrak{S}_n} \sqcup  \big\{\sigma \cdot \{z_1 + z_2 = 1\}\big\}_{\sigma \in \mathfrak{S}_n}$.

\begin{definition}
Let $\mathcal{A}$ be a \textsc{SH}--arrangement in $\mathbb{C}^n$, and $H_i$ a representative hyperplane. A representative equation of $\mathcal{A}$ is a multivariate equation
$$(E_i): c_1^{(i)} z_1 + \dots + c_{i_j}^{(i)} z_{i_j} = d^{(i)}, \quad 1 \leq i_j \leq n, \quad c_1^{(i)}, c_{i_j}^{(i)} \neq 0$$ such that for all hyperplane $H$ of a subarrangement $\displaystyle \big\{\sigma \cdot H_i \big\}_{\sigma \in \mathfrak{S}_n}$ of $\mathcal{A}$, there exists a permutation $\tau$ of $\mathfrak{S}_n$ such that $H = \{c_1^{(i)} z_{\tau(1)} + \dots + c_{i_j}^{(i)} z_{\tau(i_j)} = d^{(i)}\}$.
\end{definition}

\noindent The representative equations of $\mathcal{I}_n$ are $(E_0): z_1 = 0$, $(E_1): z_1 = 1$, and $(E_2): z_1 + z_2 = 1$.

\begin{definition}
Take a \textsc{SH}--arrangement $\displaystyle \mathcal{A} = \bigsqcup_{i=1}^r \big\{\sigma \cdot H_i \big\}_{\sigma \in \mathfrak{S}_n}$ in $\mathbb{C}^n$. For each representative hyperplane $H_i = \{c_1^{(i)} z_1 + \dots + c_{i_j}^{(i)} z_{i_j} = d^{(i)}\}$, define an equivalence relation $\sim_{H_i}$ on $\mathfrak{S}_n$ by $$\sigma \sim_{H_i} \tau \ \Longleftrightarrow \ c_1^{(i)} z_{\sigma(1)} + \dots + c_{i_j}^{(i)} z_{\sigma(i_j)} = c_1^{(i)} z_{\tau(1)} + \dots + c_{i_j}^{(i)} z_{\tau(i_j)}.$$
\end{definition}

\noindent For simplicity, we write $\mathfrak{S}_n^{\sim_{H_i}}$ for the equivalence class set instead of $\mathfrak{S}_n/\sim_{H_i}$. Then, we have
$$\mathcal{A} = \bigsqcup_{i=1}^r \big\{\sigma \cdot H_i\big\}_{\dot{\sigma} \in \mathfrak{S}_n^{\sim_{H_i}}}.$$

\noindent For the case of the hyperplane arrangement $\mathcal{I}_n$, we have 
\begin{align*}
\mathcal{I}_n =\ & \big\{\sigma \cdot \{z_1 = 0\}\big\}_{\dot{\sigma} \in \mathfrak{S}_n/\sim_{\{z_1 = 0\}}} \sqcup \big\{\sigma \cdot \{z_1 = 1\}\big\}_{\dot{\sigma} \in \mathfrak{S}_n/\sim_{\{z_1 = 1\}}} \\ & \sqcup  \big\{\sigma \cdot \{z_1 + z_2 = 1\}\big\}_{\dot{\sigma} \in \mathfrak{S}_n/\sim_{\{z_1 + z_2 = 1\}}}.
\end{align*}

\noindent Consider an equation $(E): c_1 z_1 + \dots + c_j z_j = d$ with coefficients in $\mathbb{Z}[\zeta_m]$. Denote the reduced equation modulo $q$ of $(E)$ by $(\check{E}): \check{c}_1 \check{z}_1 + \dots + \check{c}_j \check{z}_j = \check{d}$, and the solution set of the equation $(\check{E})$ in $\mathbb{F}_q[\zeta_m]^j$ by $\mathrm{Sol}(\check{E})$. For $\mathcal{I}_n$, we have
$$\mathrm{Sol}(\check{E}_0) = \{\check{0}\},\ \mathrm{Sol}(\check{E}_1) = \{\check{1}\},\ \text{and} \ \mathrm{Sol}(\check{E}_2) = \Big\{\big(\check{\frac{q+1}{2}}, \check{\frac{q+1}{2}}\big)\Big\} \sqcup \bigsqcup_{i=2}^{\frac{q-1}{2}} \big\{(\check{i}, \check{q-i+1})\big\}.$$

\begin{lemma} \label{LeSH}
Let $\mathcal{A}$ be a \textsc{SH}--arrangement in $\mathbb{C}^n$ having $r$ representative equations $(E_i): c_1^{(i)} z_1 + \dots + c_{i_j}^{(i)} z_{i_j} = d^{(i)},\ 1 \leq i_j \leq n,\ c_1^{(i)}, c_{i_j}^{(i)} \neq 0$. Then, the number of hyperplanes of the arrangement $\check{\mathcal{A}}$ an element $\check{u}$ of $\mathbb{F}_q[\zeta_m]^n$ lies on is $$h(\check{u}) = \sum_{i=1}^r \mathtt{f}\big(\mathtt{c}(\check{u}), \mathrm{Sol}(\check{E}_i)\big) \quad \text{with} \quad \mathtt{c}(\check{u}) = \big(\mathtt{o}_{\check{k}}(\check{u})\big)_{\check{k} \in \mathbb{F}_q[\zeta_m]}.$$ 
\end{lemma}

\begin{proof}
First, consider a \textsc{SH}--arrangement $\displaystyle \big\{\sigma \cdot H\big\}_{\dot{\sigma} \in \mathfrak{S}_n^{\sim_H}}$, with representative hyperplane $H$, and representative equation $(E): c_1 z_1 + \dots + c_j z_j = d$. Suppose that $\check{H}$ is the hyperplane $\{\check{c}_1 \check{z}_1 + \dots + \check{c}_j \check{z}_j = \check{d}\}$ with $\check{c}_1, \check{c}_j \neq 0$. Then
\begin{align*}
h(\check{u}) & = \big| \{ \sigma \cdot \check{H}\ |\ \dot{\sigma} \in \mathfrak{S}_n^{\sim_H},\, \check{c}_1 \check{u}_{\sigma(1)} + \dots + \check{c}_j \check{u}_{\sigma(j)} = d \} \big| \\
& = \big| \{ \dot{\sigma} \in \mathfrak{S}_n^{\sim_H}\ |\ \check{c}_1 \check{u}_{\sigma(1)} + \dots + \check{c}_j \check{u}_{\sigma(j)} = d \} \big| \\
& = \big| \{ \dot{\sigma} \in \mathfrak{S}_n^{\sim_H}\ |\ (\check{u}_{\sigma(1)}, \dots, \check{u}_{\sigma(j)}) \in  \mathrm{Sol}(\check{E}) \} \big| \\
& = \sum_{\check{z} \in \mathrm{Sol}(\check{E})} \big| \{ \dot{\sigma} \in \mathfrak{S}_n^{\sim_H}\ |\ (\check{u}_{\sigma(1)}, \dots, \check{u}_{\sigma(j)}) = \check{z} \} \big|.
\end{align*}
One can find a permutation class $\dot{\sigma}$ of $\mathfrak{S}_n^{\sim_H}$ such that $(\check{u}_{\sigma(1)}, \dots, \check{u}_{\sigma(j)}) = \check{z}$ if and only if, for all $\check{k}$ in $\mathtt{S}(\check{z})$, $\mathtt{o}_{\check{k}}(\check{u}) \geq \mathtt{o}_{\check{k}}(\check{z})$. In this case, there are $\displaystyle \binom{\mathtt{o}_{\check{k}}(\check{u})}{\mathtt{o}_{\check{k}}(\check{z})}$ possibilities to choose $\mathtt{o}_{\check{k}}(\check{z})$ coordinates of $\check{z}$ having the value $\check{k}$. Since all the elements of $\mathtt{S}(\check{z})$ should be considered, we obtain
$$\big| \{ \dot{\sigma} \in \mathfrak{S}_n^{\sim_H}\ |\ (\check{u}_{\sigma(1)}, \dots, \check{u}_{\sigma(j)}) = \check{z} \} \big| = \prod_{\check{k} \in \mathtt{S}(\check{z})} \binom{\mathtt{o}_{\check{k}}(\check{u})}{\mathtt{o}_{\check{k}}(\check{z})} = \mathtt{f}\big(\mathtt{c}(\check{u}), \check{z}\big).$$
Thus, we have $$h(\check{u}) = \sum_{\check{z} \in \mathrm{Sol}(\check{E})} \mathtt{f}\big(\mathtt{c}(\check{u}), \check{z}\big) = \mathtt{f}\big(\mathtt{c}(\check{u}), \mathrm{Sol}(\check{E})\big).$$ 
Now, let $\displaystyle \mathcal{A} = \bigsqcup_{i=1}^r \big\{\sigma \cdot H_i\big\}_{\dot{\sigma} \in \mathfrak{S}_n^{\sim_{H_i}}}$, and $h_i(\check{u})$ be the number of hyperplanes in the arrangement $\big\{\sigma \cdot H_i\big\}_{\dot{\sigma} \in \mathfrak{S}_n^{\sim_{H_i}}}$ the element $\check{u}$ lies on. Then, $$h(\check{u}) = \sum_{i=1}^r h_i(\check{u}) = \sum_{i=1}^r \mathtt{f}\big(\mathtt{c}(\check{u}), \mathrm{Sol}(\check{E}_i)\big).$$
\end{proof}

\noindent Let $(E_i)$ be the $r$ representative equations of a \textsc{SH}--arrangement. We partition the set $\displaystyle \bigcup_{i=1}^r \mathrm{Sol}(\check{E}_i)$ into subsets $\mathrm{M}_1, \dots, \mathrm{M}_s$ having the following properties:
\begin{itemize}
\item[$\bullet$] For every element $\check{z}$ in $\mathrm{M}_i$, there exists another element $\check{u}$ in $\mathrm{M}_i$ such that $\mathtt{S}(\check{z}) \cap \mathtt{S}(\check{u}) \neq \emptyset$.
\item[$\bullet$] If $i \neq j$, then, for every $\check{z}$ in $\mathrm{M}_i$, and $\check{u}$ in $\mathrm{M}_j$, we have $\mathtt{S}(\check{z}) \cap \mathtt{S}(\check{u}) = \emptyset$.
\end{itemize}
The partition $\mathrm{M}_1, \dots, \mathrm{M}_s$ is called indice partition of $\displaystyle \bigcup_{i=1}^r \mathrm{Sol}(\check{E}_i)$. 

\noindent For $\mathcal{I}_n$, the indice partition of $\mathrm{Sol}(\check{E}_0) \cup \mathrm{Sol}(\check{E}_1) \cup \mathrm{Sol}(\check{E}_2)$ is
$$M_0 = \big\{\check{0}, \check{1}, (\check{0}, \check{1})\big\},\ M_1 = \Big\{\big(\check{\frac{q+1}{2}}, \check{\frac{q+1}{2}}\big)\Big\},\ \text{and}\ \displaystyle M_i = \big\{(\check{i}, \check{q-i+1})\big\}\ \text{for}\ i \in \big\{2, \dots, \frac{q-1}{2}\big\}.$$

\begin{proposition} \label{PrCob}
Let $\mathcal{A}$ be a \textsc{SH}--arrangement having representative equations whose solution indice partition is $\mathrm{M}_1, \dots, \mathrm{M}_s$. Then its $\zeta_m$--coboundary polynomial is
$$\bar{\chi}_{\mathcal{A}}^{(\zeta_m)}(q,t) = q^{l_m(r(\mathcal{A})-n)} \sum_{\substack{\mathsf{a} \in \mathbb{N}^{\mathbb{F}_q[\zeta_m]} \\ \|\mathsf{a}\| = n}} \binom{n}{\mathsf{a}} t^{\sum_{j=1}^s \mathtt{f}(\mathsf{a}, \mathrm{M}_j)}.$$
\end{proposition}

\begin{proof}
Using the indice partition of $\displaystyle \bigcup_{i=1}^r \mathrm{Sol}(\check{E}_i)$, the number of hyperplanes the element $\check{u}$ of lies on is $$h(\check{u}) = \sum_{j=1}^s \mathtt{f}\big(\mathtt{c}(\check{u}), \mathrm{M}_j\big) \quad \text{(Lemma \ref{LeSH})}.$$
From Theorem \ref{ThCo}, we have
\begin{align*}
q^{l_m(n-r(\mathcal{A}))} \bar{\chi}_{\mathcal{A}}^{(\zeta_m)}(q,t) & = \sum_{\check{u} \in \mathbb{F}_q[\zeta_m]^n} t^{\sum_{j=1}^s \mathtt{f}\big(\mathtt{c}(\check{u}), \mathrm{M}_j\big)} \\
& = \sum_{\substack{\mathsf{a} \in \mathbb{N}^{\mathbb{F}_q[\zeta_m]} \\ \|\mathsf{a}\| = n}} \binom{n}{\mathsf{a}} t^{\sum_{j=1}^s \mathtt{f}(\mathsf{a}, \mathrm{M}_j)}.
\end{align*}
\end{proof}

\noindent With $r(\mathcal{I}_n)=n$, the coboundary polynomial of $\mathcal{I}_n$ is
$$\bar{\chi}_{\mathcal{I}_n}(q,t) = \sum_{a_0 + \dots + a_{q-1} = n} \binom{n}{a_0, \dots, a_{q-1}} t^{a_0 + a_1 + a_0 a_1}\ t^{\displaystyle \binom{a_{\frac{q+1}{2}}}{2}}\ \prod_{i=2}^{\frac{q-1}{2}} t^{a_i a_{q-i+1}}.$$

\noindent Take $r$ hyperplanes $H_i = \{c_1^{(i)} z_1 + \dots + c_{i_j}^{(i)} z_{i_j} = d^{(i)}\}$ with $c_{i_j}^{(i)} \neq 0$. Setting $h = \max(i_1, \dots, i_r)$, define a sequence of \textsc{SH}--arrangements with representative hyperplanes $H_1, \dots, H_r$ by a sequence of hyperplane arrangements $(\mathcal{A}_n)_{n \geq h}$ in $(\mathbb{C}^n)_{n \geq h}$ such that $$\mathcal{A}_n = \bigsqcup_{i=1}^r  \big\{\sigma \cdot H_i \big\}_{\dot{\sigma} \in \mathfrak{S}_n^{\sim_{H_i}}}.$$

\begin{proposition} \label{PrSy}
Let $(\mathcal{A}_n)_{n \geq h}$ be a sequence of \textsc{SH}--arrangements having the same representative equations, and $\mathrm{M}_1, \dots, \mathrm{M}_s$ their common solution indice partition. Consider the EGF
$$g(x) = \sum_{n \in \mathbb{N}} g_n \frac{x^n}{n!} \quad \text{with} \quad g(x) = \prod_{j=1}^s \ \sum_{n \in \mathbb{N}} \Big( \sum_{\substack{\mathsf{a} \in \mathbb{N}^{\mathtt{S}(\mathrm{M}_j)} \\ \|\mathsf{a}\| = n}} \binom{n}{\mathsf{a}} t^{\mathtt{f}(\mathsf{a}, \mathrm{M}_j)} \Big) \frac{x^n}{n!}.$$
Then, $$\sum_{n \geq h} q^{l_m(n-r(\mathcal{A}_n))} \bar{\chi}_{\mathcal{A}_n}^{(\zeta_m)}(q,t) \frac{x^n}{n!} = g(x) - \sum_{n=0}^{h-1} g_n \frac{x^n}{n!}.$$
\end{proposition}

\begin{proof}
From Proposition \ref{PrCob}, we know that $$q^{l_m(n-r(\mathcal{A}_n))} \bar{\chi}_{\mathcal{A}_n}^{(\zeta_m)}(q,t) = \sum_{\substack{\mathsf{a} \in \mathbb{N}^{\mathbb{F}_q[\zeta_m]} \\ \|\mathsf{a}\| = n}} \binom{n}{\mathsf{a}} t^{\sum_{j=1}^s \mathtt{f}(\mathsf{a}, \mathrm{M}_j)}.$$
Set the EGF $\displaystyle g(x) = \sum_{n \in \mathbb{N}} g_n \frac{x^n}{n!}$ with
$\displaystyle g_n = \sum_{\substack{\mathsf{a} \in \mathbb{N}^{\mathbb{F}_q[\zeta_m]} \\ \|\mathsf{a}\| = n}} \binom{n}{\mathsf{a}} t^{\sum_{j=1}^s \mathtt{f}(\mathsf{a}, \mathrm{M}_j)}$. We deduce from Proposition~\ref{PrPrin} that $$g(x) = \prod_{j=1}^s \ \sum_{n \in \mathbb{N}} \Big( \sum_{\substack{\mathsf{a} \in \mathbb{N}^{\mathtt{S}(\mathrm{M}_j)} \\ \|\mathsf{a}\| = n}} \binom{n}{\mathsf{a}} t^{\mathtt{f}(\mathsf{a}, \mathrm{M}_j)} \Big) \frac{x^n}{n!}.$$
Since $\bar{\chi}_{\mathcal{A}_n}^{(\zeta_m)}(q,t)$ is not defined for $n < h$, and $q^{l_m(n-r(\mathcal{A}_n))} \bar{\chi}_{\mathcal{A}_n}^{(\zeta_m)}(q,t) = g_n$ for $n \geq h$, we get the result.
\end{proof}

\noindent Letting $\mathcal{I}_1 = \big\{\{0\}, \{1\}\big\}$, the EGF of the coboundary polynomials of $\mathcal{I}_n$ is
\begin{align*}
\sum_{n \in \mathbb{N}} q \bar{\chi}_{\mathcal{I}_n}(q,t) \frac{x^n}{n!}\ =\ & \Big( \sum_{n \in \mathbb{N}} t^{\binom{n}{2}} \frac{x^n}{n!} \Big) \times \Big( \sum_{n \in \mathbb{N}} \big( \sum_{a+b=n} \binom{n}{a,b} t^{a+b+ab} \big) \frac{x^n}{n!} \Big) \\
& \times \Big( \sum_{n \in \mathbb{N}} \big( \sum_{a+b=n} \binom{n}{a,b} t^{ab} \big) \frac{x^n}{n!} \Big)^{\frac{q-3}{2}}.
\end{align*}

\section{The $\zeta_m$--Coboundary Polynomial of a \textsc{CSH}--Arrangement} \label{Colored}

\noindent This section describes the most interesting part of this article, namely the $\zeta_m$--coboundary polynomial of a \textsc{CSH}--arrangement. We deduce the exponential generating function of a sequence of \textsc{CSH}--arrangement coboundary polynomials. Then, we apply the results on the hyperplane arrangement $\mathcal{G}_{m,p,n}$. Shephard and Todd extended in 1954 the concept of reflection in an Euclidean space to reflection in a Hermitian space \cite{ShTo}. Recall that a complex reflection $r_{a, \alpha}$ of order $m$ is a linear transformation on $\mathbb{C}^n$ which sends some nonzero vector $a$ to $\alpha a$, where $\alpha$ is a primitive $m^{\text{th}}$ root of unity, while fixing pointwise the hyperplane $H_a = \mathrm{Fix}\,r_{a, \alpha}$ orthogonal to $a$. It has the formula $\displaystyle r_{a, \alpha} := v - (1 - \alpha) \frac{(v,a)}{(a,a)}a$. The complex reflection groups are the finite groups generated by complex reflections. They naturally contain the real reflection groups, and have a wide range of applications, including knot theory, Hecke algebras, and differential equations \cite{St}. We focus on the imprimitive reflection groups, or the complex reflection groups whose action representations permute subspaces among themselves, which are the groups $G(m,p,n)$ with $p$ dividing $m$.

\smallskip

\noindent For the presentations and explanations, we use both colored permutations of $\mathbb{U}_m \wr \mathfrak{S}_n$
$$\pi = \left( \begin{array}{cccc} 1 & 2 & \dots & n\\
\xi_1 \sigma(1) & \xi_2 \sigma(2) & \dots & \xi_n \sigma(n) \end{array} \right)\quad \text{and}\quad
\phi = \left( \begin{array}{cccc} 1 & 2 & \dots & n\\
\nu_1 \tau(1) & \nu_2 \tau(2) & \dots & \nu_n \tau(n) \end{array} \right).$$

\begin{definition}
We can choose some hyperplanes $H_1, \dots, H_r$ of a \textsc{CSH}--arrangement $\mathcal{A}$ in $\mathbb{C}^n$ to write it in the form
$$\mathcal{A} = \bigsqcup_{i=1}^r \big\{\pi \cdot H_i\big\}_{\pi \in \mathbb{U}_m \wr \mathfrak{S}_n}.$$ 
The hyperplanes $H_1, \dots, H_r$ are called representative hyperplanes for $\mathcal{A}$.
\end{definition}

\noindent Recall that the hyperplane arrangement associated to $G(m,p,n)$ is
$$\mathcal{G}_{m,p,n} = \big\{\pi \cdot \{z_i = 0\}\big\}_{\pi \in \mathbb{U}_m \wr \mathfrak{S}_n} \sqcup \big\{\pi \cdot \{z_i - z_j = 0\}\big\}_{\pi \in \mathbb{U}_m \wr \mathfrak{S}_n}.$$

\begin{definition}
\noindent Let $\mathcal{A}$ be a \textsc{CSH}--arrangement in $\mathbb{C}^n$, and $H_i$ a representative hyperplane. A representative equation of $\mathcal{A}$ associated to $H_i$ is a multivariate equation
$$(E_i): c_1^{(i)} z_1 + \dots + c_{i_j}^{(i)} z_{i_j} = d^{(i)}, \quad 1 \leq i_j \leq n, \quad c_1^{(i)}, c_{i_j}^{(i)} \neq 0$$ such that for all hyperplane $H$ of a subarrangement $\{\pi \cdot H_i\}_{\pi \in \mathbb{U}_m \wr \mathfrak{S}_n}$ of $\mathcal{A}$, there exists a colored permutation $\phi$ such that
$H = \{c_1^{(i)} \nu_1 z_{\tau(1)} + \dots + c_{i_j}^{(i)} \nu_{i_j} z_{\tau(i_j)} = d^{(i)}\}$.
\end{definition}

\noindent The representative equations of $\mathcal{G}_{m,p,n}$ are $(E_1): z_1 = 0$, and $(E_2): z_1 - z_2 = 0$.

\begin{definition}
Take a \textsc{CSH}--arrangement $\displaystyle \mathcal{A} = \bigsqcup_{i=1}^r \big\{\pi \cdot H_i\big\}_{\pi \in \mathbb{U}_m \wr \mathfrak{S}_n}$ in $\mathbb{C}^n$. For each representative hyperplane $H_i = \{c_1^{(i)} z_1 + \dots + c_{i_j}^{(i)} z_{i_j} = d^{(i)}\}$, we define an equivalence relation $\sim_{H_i}$ on $\mathbb{U}_m \wr \mathfrak{S}_n$ by
$$\pi \sim_{H_i} \phi \ \Longleftrightarrow \ c_1^{(i)} \xi_1 z_{\sigma(1)} + \dots + c_{i_j}^{(i)} \xi_{i_j} z_{\sigma(i_j)} = c_1^{(i)} \nu_1 z_{\tau(1)} + \dots + c_{i_j}^{(i)} \nu_{i_j} z_{\tau(i_j)}.$$
For simplicity, we write $\mathbb{U}_m \wr \mathfrak{S}_n^{\sim_{H_i}}$ for the equivalence class set instead of $\mathbb{U}_m \wr \mathfrak{S}_n/\sim_{H_i}$. Then, we have $$\mathcal{A} = \bigsqcup_{i=1}^r  \big\{\pi \cdot H_i \big\}_{\dot{\pi} \in \mathbb{U}_m \wr \mathfrak{S}_n^{\sim_{H_i}}}.$$
\end{definition}

\noindent For the hyperplane arrangement associated to $G(m,p,n)$, we have
$$\mathcal{G}_{m,p,n} = \big\{\pi \cdot \{z_i = 0\}\big\}_{\dot{\pi} \in \mathbb{U}_m \wr \mathfrak{S}_n^{\sim_{\{z_i = 0\}}}} \sqcup \big\{\pi \cdot \{z_i - z_j = 0\}\big\}_{\dot{\pi} \in \mathbb{U}_m \wr \mathfrak{S}_n^{\sim_{\{z_i - z_j = 0\}}}}.$$

\noindent Define the set $\mathbb{U}_m \ast v := \big\{ (\xi_1 v_1, \dots, \xi_n v_n)\ |\ \xi_1, \dots, \xi_n \in \mathbb{U}_m \big\}$. More generally, for a subset $M$ of $\displaystyle \bigsqcup_{n \in \mathbb{N}} \mathbb{S}^n$, define the set $$\mathbb{U}_m \ast M := \bigcup_{v \in M} \mathbb{U}_m \ast v.$$

\begin{lemma} \label{LeEq}
Consider a \textsc{CSH}--arrangement $\mathcal{A}$ in $\mathbb{C}^n$ with representative hyperplanes $H_1, \dots, H_r$. Take a vector $\check{u}$ of $\mathbb{F}_q[\zeta_m]^n$. Then,
$$\forall \check{x}, \check{y} \in \mathbb{U}_m \ast \check{u},\ h(\check{x}) = h(\check{y}).$$
\end{lemma}

\begin{proof}
Denote by $\mathcal{A}(\check{x})$ resp. $\mathcal{A}(\check{y})$ the set of hyperplanes in $\mathcal{A}$ containing $\check{x} = (\check{x}_1, \dots, \check{x}_n)$ resp. $\check{y} = (\check{y}_1, \dots, \check{y}_n)$. There are elements $\xi_i$ of $\mathbb{U}_m$ such that $\check{y}_i = \xi_i \check{x}_i$. Since the following function 
$$g_{\check{x}, \check{y}}: \left. \begin{array}{ccc} \mathcal{A}(\check{x}) & \rightarrow & \mathcal{A}(\check{y}) \\ 
\pi \cdot H_i & \mapsto & \left( \begin{array}{cccc} 1 & 2 & \dots & n\\
\bar{\xi}_1 1 & \bar{\xi}_2 2 & \dots & \bar{\xi}_n n \end{array} \right) \cdot \pi \cdot H_i \end{array} \right.,$$
where $\bar{\xi}_i$ is the conjugate of $\xi$, is bijective, then $|\mathcal{A}(\check{x})| = |\mathcal{A}(\check{y})|$.
\end{proof}

\noindent Lemma \ref{LeEq} leads us to define the equivalence relation $\sim_{\ast}$ on $\mathbb{F}_q[\zeta_m]^n$ by
$$\check{x} \sim_{\ast} \check{y} \ \Longleftrightarrow \ \check{x} \in \mathbb{U}_m \ast \check{y}.$$
We write $\ddot{x}$ for the set $\{\check{y} \in \mathbb{F}_q[\zeta_m]^n\ |\ \check{x} \sim_{\ast} \check{y}\}$, and $\mathbb{F}_q[\zeta_m]_{\ast}^n$ for the equivalence class set $\mathbb{F}_q[\zeta_m]^n/\sim_{\ast}$. More generally, for a subset $V$ of $\displaystyle \bigsqcup_{n \in \mathbb{N}} \mathbb{F}_q[\zeta_m]^n$, define the set $$V_{\ast} := \{\ddot{v}\ |\ v \in V\}.$$

\noindent For $\mathcal{G}_{m,p,n}$, we have $\mathrm{Sol}(\check{E}_1) = \{\check{0}\}$ and $\mathrm{Sol}(\check{E}_2) = \big\{ (\check{k}, \check{k})\ |\ \check{k} \in \mathbb{F}_q[\zeta_m] \big\}$, that means
$$\mathrm{Sol}(\check{E}_1)_{\ast} = \{\ddot{0}\} \quad \text{and} \quad \mathrm{Sol}(\check{E}_2)_{\ast} = \big\{ (\ddot{k}, \ddot{k})\ |\ \ddot{k} \in \mathbb{F}_q[\zeta_m]_{\ast} \big\}.$$

\begin{proposition} \label{PrCSH}
Let $\mathcal{A}$ be a \textsc{CSH}--arrangement in $\mathbb{C}^n$ having $r$ representative equations $(E_i): c_1^{(i)} z_1 + \dots + c_{i_j}^{(i)} z_{i_j} = d^{(i)},\ 1 \leq i_j \leq n,\ c_1^{(i)}, c_{i_j}^{(i)} \neq 0$. Then, the number of hyperplanes of the arrangement $\check{\mathcal{A}}$ an element $\check{u}$ of $\mathbb{F}_q[\zeta_m]^n$ lies on is
$$h(\check{u})=\sum_{i=1}^r \mathtt{f}_{\ddot{0}}^{(m)}\big(\mathtt{c}(\ddot{u}), \mathrm{Sol}(\check{E}_i)_{\ast}\big) \quad \text{with} \quad \mathtt{c}(\ddot{u}) = \big(\mathtt{o}_{\ddot{k}}(\ddot{u})\big)_{\ddot{k} \in \mathbb{F}_q[\zeta_m]_{\ast}}.$$ 
\end{proposition}

\begin{proof}
Consider first a \textsc{CSH}--arrangement $\{\pi \cdot H\}_{\dot{\pi} \in \mathbb{U}_m \wr \mathfrak{S}_n^{\sim_H}}$, with representative hyperplane $H$, and representative equation $(E): c_1 z_1 + \dots + c_j z_j = d$. Suppose that $\check{H}$ is the hyperplane $\{\check{c}_1 \check{z}_1 + \dots + \check{c}_j \check{z}_j = \check{d}\}$ with $\check{c}_1, \check{c}_j \neq 0$. Then
\begin{align*}
h(\check{u}) & = \big| \{ \pi \cdot \check{H}\ |\ \dot{\pi} \in \mathbb{U}_m \wr \mathfrak{S}_n^{\sim_H},\, \check{c}_1 \xi_1 \check{u}_{\sigma(1)} + \dots + \check{c}_j \xi_j \check{u}_{\sigma(j)} = d \} \big| \\
& = \big| \{ \dot{\pi} \in \mathbb{U}_m \wr \mathfrak{S}_n^{\sim_H}\ |\ \check{c}_1 \xi_1 \check{u}_{\sigma(1)} + \dots + \check{c}_j \xi_j \check{u}_{\sigma(j)} = d \} \big| \\
& = \big| \{ \dot{\pi} \in \mathbb{U}_m \wr \mathfrak{S}_n^{\sim_H}\ |\ (\xi_1 \check{u}_{\sigma(1)}, \dots, \xi_j \check{u}_{\sigma(j)}) \in  \mathrm{Sol}(\check{E}) \} \big| \\
& = \sum_{\check{z} \in \mathrm{Sol}(\check{E})} \big| \{ \dot{\pi} \in \mathbb{U}_m \wr \mathfrak{S}_n^{\sim_H}\ |\ (\xi_1 \check{u}_{\sigma(1)}, \dots, \xi_j \check{u}_{\sigma(j)}) = \check{z} \} \big|.
\end{align*}
One can find a colored permutation class $\dot{\pi}$ of $\mathbb{U}_m \wr \mathfrak{S}_n^{\sim_H}$ such that $(\xi_1 \check{u}_{\sigma(1)}, \dots, \xi_j \check{u}_{\sigma(j)}) = \check{z}$ if and only if, for all $\check{k}$ in $\mathtt{S}(\check{z})$, $\mathtt{o}_{\ddot{k}}(\ddot{u}) \geq \mathtt{o}_{\ddot{k}}(\ddot{z})$. In this case, there are $\displaystyle \binom{\mathtt{o}_{\ddot{k}}(\ddot{u})}{\mathtt{o}_{\ddot{k}}(\ddot{z})}$ possibilities to choose $\mathtt{o}_{\ddot{k}}(\ddot{z})$ coordinates of $\ddot{z}$ having the value $\ddot{k}$.
\begin{itemize}
\item[$\bullet$] If $\mathtt{S}(\check{z}) = \{\check{0}\}$, then
$$\big| \{ \dot{\pi} \in \mathbb{U}_m \wr \mathfrak{S}_n^{\sim_H}\ |\ (\xi_1 \check{u}_{\sigma(1)}, \dots, \xi_j \check{u}_{\sigma(j)}) = \check{z} \} \big| = m^{j-1} \binom{\mathtt{o}_{\ddot{0}}(\ddot{u})}{\mathtt{o}_{\ddot{0}}(\ddot{z})} = \mathtt{f}_{\ddot{0}}^{(m)}\big(\mathtt{c}(\ddot{u}), \ddot{z}\big).$$
\item[$\bullet$] If $\mathtt{S}(\check{z}) \neq \{\check{0}\}$, then
$$\big| \{ \dot{\pi} \in \mathbb{U}_m \wr \mathfrak{S}_n^{\sim_H}\ |\ (\xi_1 \check{u}_{\sigma(1)}, \dots, \xi_j \check{u}_{\sigma(j)}) = \check{z} \} \big| = m^{\mathtt{o}_{\ddot{0}}(\ddot{z})} \prod_{\ddot{k} \in \mathtt{S}(\ddot{z})} \binom{\mathtt{o}_{\ddot{k}}(\ddot{u})}{\mathtt{o}_{\ddot{k}}(\ddot{z})} = \mathtt{f}_{\ddot{0}}^{(m)}\big(\mathtt{c}(\ddot{u}), \ddot{z}\big).$$
\end{itemize}
Thus, we have $$h(\check{u}) = \sum_{\check{z} \in \mathrm{Sol}(\check{E})} \mathtt{f}_{\ddot{0}}^{(m)}\big(\mathtt{c}(\ddot{u}), \ddot{z}\big) = \mathtt{f}_{\ddot{0}}^{(m)}\big(\mathtt{c}(\ddot{u}), \mathrm{Sol}(\check{E})_{\ast}\big).$$ 
Now, let $\displaystyle \mathcal{A} = \bigsqcup_{i=1}^r \big\{\pi \cdot H_i\big\}_{\dot{\pi} \in \mathbb{U}_m \wr \mathfrak{S}_n^{\sim_{H_i}}}$, and $h_i(\check{u})$ be the number of hyperplanes in the arrangement $\big\{\pi \cdot H_i\big\}_{\dot{\pi} \in \mathbb{U}_m \wr \mathfrak{S}_n^{\sim_{H_i}}}$ the element $\check{u}$ lies on. Then, $$h(\check{u}) = \sum_{i=1}^r h_i(\check{u}) = \sum_{i=1}^r \mathtt{f}_{\ddot{0}}^{(m)}\big(\mathtt{c}(\ddot{u}), \mathrm{Sol}(\check{E}_i)_{\ast}\big).$$
\end{proof}

\noindent In analogy with $\mathrm{M}_1, \dots, \mathrm{M}_s$ and $\displaystyle \bigcup_{i=1}^r \mathrm{Sol}(\check{E}_i)$, define the indice partition $\mathrm{M}_{1 \ast}, \dots, \mathrm{M}_{s \ast}$ of the set $\displaystyle \bigcup_{i=1}^r \mathrm{Sol}(\check{E}_i)_{\ast}$ as follows:
\begin{itemize}
\item[$\bullet$] For every element $\ddot{z}$ in $\mathrm{M}_{i \ast}$, there is another element $\ddot{u}$ in $\mathrm{M}_{i \ast}$ such that $\mathtt{S}(\ddot{z}) \cap \mathtt{S}(\ddot{u}) \neq \emptyset$.
\item[$\bullet$] If $i \neq j$, then, for every $\ddot{z}$ in $\mathrm{M}_{i \ast}$, and $\ddot{u}$ in $\mathrm{M}_{j \ast}$, we have $\mathtt{S}(\ddot{z}) \cap \mathtt{S}(\ddot{u}) = \emptyset$.
\end{itemize}

\noindent For $\mathcal{G}_{m,p,n}$, the indice partition of $\mathrm{Sol}(\check{E}_1)_{\ast} \cup \mathrm{Sol}(\check{E}_2)_{\ast}$ is $$\mathrm{M}_{\check{0} \ast} = \big\{\ddot{0},\, (\ddot{0}, \ddot{0})\big\},\ \text{and}\ \mathrm{M}_{\check{k} \ast} = \big\{(\ddot{k}, \ddot{k})\big\}\ \text{for}\ \ddot{k} \in \mathbb{F}_q[\zeta_m]_{\ast} \setminus \{\ddot{0}\}.$$

\noindent Now, we come to the main result of this section.

\begin{theorem} \label{ThCSH}
Let $\mathcal{A}$ be a \textsc{CSH}--arrangement having representative equations whose indice partition is $\mathrm{M}_{1 \ast}, \dots, \mathrm{M}_{s \ast}$. Then its $\zeta_m$--coboundary polynomial is
$$\bar{\chi}_{\mathcal{A}}^{(\zeta_m)}(q,t) = q^{l_m(r(\mathcal{A})-n)} \sum_{\substack{\mathsf{a} \in \mathbb{N}^{\mathbb{F}_q[\zeta_m]_{\ast}} \\ \|\mathsf{a}\| = n}} \binom{n}{\mathsf{a}} m^{n- \mathsf{a}_{\ddot{0}}} t^{\sum_{j=1}^s \mathtt{f}_{\ddot{0}}^{(m)}(\mathsf{a}, \mathrm{M}_{j \ast})}.$$
\end{theorem}

\begin{proof}
Using the indice partition of $\displaystyle \bigcup_{i=1}^r \mathrm{Sol}(\check{E}_i)$, the number of hyperplanes the element $\check{u}$ of lies on is $$h(\check{u}) = \sum_{j=1}^s \mathtt{f}_{\ddot{0}}^{(m)}\big(\mathtt{c}(\ddot{u}), \mathrm{M}_{j \ast}\big) \quad \text{(Proposition \ref{PrCSH})}.$$
From Theorem \ref{ThCo}, we have
\begin{align*}
q^{l_m(n-r(\mathcal{A}))} \bar{\chi}_{\mathcal{A}}^{(\zeta_m)}(q,t) & = \sum_{\check{u} \in \mathbb{F}_q[\zeta_m]^n} t^{\sum_{j=1}^s \mathtt{f}_{\ddot{0}}^{(m)}\big(\mathtt{c}(\ddot{u}), \mathrm{M}_{j \ast}\big)} \\
& = \sum_{\ddot{u} \in \mathbb{F}_q[\zeta_m]_{\ast}^n} m^{n- \mathtt{o}_{\check{0}}(\check{u})} t^{\sum_{j=1}^s \mathtt{f}_{\ddot{0}}^{(m)}\big(\mathtt{c}(\ddot{u}), \mathrm{M}_{j \ast}\big)} \\
& \qquad \text{(Lemma \ref{LeEq} with $|\mathbb{U}_m \ast \check{u}| = m^{n- \mathtt{o}_{\check{0}}(\check{u})}$)} \\
& = \sum_{\substack{\mathsf{a} \in \mathbb{N}^{\mathbb{F}_q[\zeta_m]_{\ast}} \\ \|\mathsf{a}\| = n}} \binom{n}{\mathsf{a}} m^{n- \mathsf{a}_{\ddot{0}}} t^{\sum_{j=1}^s \mathtt{f}_{\ddot{0}}^{(m)}(\mathsf{a}, \mathrm{M}_{j \ast})}.
\end{align*}
\end{proof}

\noindent With $r(\mathcal{G}_{m,p,n}) = n$, the $\zeta_m$--coboundary polynomial of $\mathcal{G}_{m,p,n}$ is $$\bar{\chi}_{\mathcal{G}_{m,p,n}}^{(\zeta_m)}(q,t) = \sum_{\substack{\mathsf{a} \in \mathbb{N}^{\mathbb{F}_q[\zeta_m]_{\ast}} \\ \|\mathsf{a}\| = n}} \binom{n}{\mathsf{a}} m^{n- \mathsf{a}_{\ddot{0}}} t^{\mathsf{a}_{\ddot{0}} + m \binom{\mathsf{a}_{\ddot{0}}}{2}} t^{\sum_{\ddot{k} \in \mathbb{F}_q[\zeta_m]_{\ast} \setminus \{\ddot{0}\}} \binom{\mathsf{a}_{\ddot{k}}}{2}}.$$

\noindent Take $r$ hyperplanes $H_i = \{c_1^{(i)} z_1 + \dots + c_{i_j}^{(i)} z_{i_j} = d^{(i)}\}$ with $c_{i_j}^{(i)} \neq 0$. Setting $h = \max(i_1, \dots, i_r)$, define a sequence of \textsc{CSH}--arrangements with representative hyperplanes $H_1, \dots, H_r$ by a sequence of hyperplane arrangements $(\mathcal{A}_n)_{n \geq h}$ in $(\mathbb{C}^n)_{n \geq h}$ such that $$\mathcal{A}_n = \bigsqcup_{i=1}^r  \big\{\pi \cdot H_i\big\}_{\dot{\pi} \in \mathbb{U}_m \wr \mathfrak{S}_n^{\sim_{H_i}}}.$$

\begin{theorem} \label{ThCoEGF}
Let $(\mathcal{A}_n)_{n \geq h}$ be a sequence of \textsc{CSH}--arrangements having the same representative equations, and $\mathrm{M}_{1 \ast}, \dots, \mathrm{M}_{s \ast}$ their common indice partition. Consider the EGF 
$$g(x) = \sum_{n \in \mathbb{N}} g_n \frac{x^n}{n!} \quad \text{with} \quad g(x) = \prod_{j=1}^s \ \sum_{n \in \mathbb{N}} \Big( \sum_{\substack{\mathsf{a} \in \mathbb{N}^{\mathtt{S}(\mathrm{M}_{j \ast})} \\ \|\mathsf{a}\| = n}} \binom{n}{\mathsf{a}} m^{n- \mathsf{a}_{\ddot{0}}} t^{\mathtt{f}_{\ddot{0}}^{(m)}(\mathsf{a}, \mathrm{M}_{j \ast})} \Big) \frac{x^n}{n!}.$$
such that $\mathsf{a}_{\ddot{0}} = 0$ if $\ddot{0} \notin \mathtt{S}(\mathrm{M}_{j \ast})$. Then,
$$\sum_{n \geq h} q^{l_m(n-r(\mathcal{A}_n))} \bar{\chi}_{\mathcal{A}_n}^{(\zeta_m)}(q,t) \frac{x^n}{n!} = g(x) - \sum_{n=0}^{h-1} g_n \frac{x^n}{n!}.$$
\end{theorem}

\begin{proof}
From Theorem \ref{ThCSH}, we know that $$q^{l_m(n- r(\mathcal{A}_n))} \bar{\chi}_{\mathcal{A}_n}^{(\zeta_m)}(q,t) = \sum_{\substack{\mathsf{a} \in \mathbb{N}^{\mathbb{F}_q[\zeta_m]_{\ast}} \\ \|\mathsf{a}\| = n}} \binom{n}{\mathsf{a}} m^{n- \mathsf{a}_{\ddot{0}}} t^{\sum_{j=1}^s \mathtt{f}_{\ddot{0}}^{(m)}(\mathsf{a}, \mathrm{M}_{j \ast})}.$$
Set the EGF $\displaystyle g(x) = \sum_{n \in \mathbb{N}} g_n \frac{x^n}{n!}$ with
$\displaystyle g_n = \sum_{\substack{\mathsf{a} \in \mathbb{N}^{\mathbb{F}_q[\zeta_m]_{\ast}} \\ \|\mathsf{a}\| = n}} \binom{n}{\mathsf{a}} m^{n- \mathsf{a}_{\ddot{0}}} t^{\sum_{j=1}^s \mathtt{f}_{\ddot{0}}^{(m)}(\mathsf{a}, \mathrm{M}_{j \ast})}$. We deduce from Proposition \ref{PrPrin} that $$g(x) = \prod_{j=1}^s \ \sum_{n \in \mathbb{N}} \Big( \sum_{\substack{\mathsf{a} \in \mathbb{N}^{\mathtt{S}(\mathrm{M}_{j \ast})} \\ \|\mathsf{a}\| = n}} \binom{n}{\mathsf{a}} m^{n- \mathsf{a}_{\ddot{0}}} t^{\mathtt{f}_{\ddot{0}}^{(m)}(\mathsf{a}, \mathrm{M}_{j \ast})} \Big) \frac{x^n}{n!}.$$
Since $\bar{\chi}_{\mathcal{A}_n}^{(\zeta_m)}(q,t)$ is not defined for $n < h$, and $q^{l_m(n- r(\mathcal{A}_n))}\bar{\chi}_{\mathcal{A}_n}^{(\zeta_m)}(q,t) = g_n$ for $n \geq h$, we get the result.
\end{proof}

\noindent Since $|\mathbb{F}_q[\zeta_m]_{\ast} \setminus \{\ddot{0}\}| = \frac{q^{l_m}-1}{m}$, the EGF of the $\zeta_m$--coboundary polynomials of $(\mathcal{G}_{m,p,n})_{n \in \mathbb{N}}$ is
$$\sum_{n \in \mathbb{N}} \bar{\chi}_{\mathcal{G}_{m,p,n}}^{(\zeta_m)}(q,t) \frac{x^n}{n!} = \Big(\sum_{n \in \mathbb{N}} t^{n+m\binom{n}{2}} \frac{x^n}{n!}\Big) \Big(\sum_{n \in \mathbb{N}} m^n t^{\binom{n}{2}} \frac{x^n}{n!}\Big)^{\frac{q^{l_m}-1}{m}}.$$

\noindent If $p = m$, the hyperplane arrangement associated to $G(m,m,n)$ is $$\mathcal{G}_{m,m,n} = \big\{\pi \cdot \{z_i - z_j = 0\}\big\}_{\dot{\pi} \in \mathbb{U}_m \wr \mathfrak{S}_n^{\sim_{\{z_i - z_j = 0\}}}}.$$ And, for $m \geq 2$, the EGF of the $\zeta_m$--coboundary polynomials associated to $(\mathcal{G}_{m,m,n})_{n \in \mathbb{N}}$ is
$$\sum_{n \in \mathbb{N}} \bar{\chi}_{\mathcal{G}_{m,m,n}}^{(\zeta_m)}(q,t) \frac{x^n}{n!} = \Big(\sum_{n \in \mathbb{N}} t^{m\binom{n}{2}} \frac{x^n}{n!}\Big) \Big(\sum_{n \in \mathbb{N}} m^n t^{\binom{n}{2}} \frac{x^n}{n!}\Big)^{\frac{q^{l_m}-1}{m}}.$$

\bibliographystyle{abbrvnat}

\end{document}